\documentclass[a4paper,12pt, reqno]{amsart}

\usepackage{amsmath, amsthm, amscd, amsfonts, amssymb, graphicx, color}
\usepackage[bookmarksnumbered, colorlinks, plainpages]{hyperref}

\usepackage{enumitem}
\usepackage{url}
\usepackage{multirow}
\usepackage{graphicx}
\usepackage{subfigure}
\usepackage[pass]{geometry}
\usepackage{cite}
\usepackage{cancel}
\usepackage{color}
\usepackage{pgf,tikz}
\usetikzlibrary{arrows}
\usepackage{lscape}
\tikzstyle{v} = [circle, draw, inner sep=2pt, minimum size=3pt, fill=black]
\usetikzlibrary{calc,shapes,decorations.pathreplacing,matrix}
\usetikzlibrary{patterns}

\tikzset{square matrix/.style={
    matrix of nodes,
    column sep=-\pgflinewidth, row sep=-\pgflinewidth,
    nodes={draw,
      minimum height=4.5pt,
      anchor=center,
      text width=4.5pt,
      align=center,
      inner sep=0pt
    },
  },
  square matrix/.default=1.2cm
}

\newtheorem{Theorem}{Theorem}[section]
\newtheorem{Definition}[Theorem]{Definition}

\newtheorem{Proposition}[Theorem]{Proposition}
\newtheorem{Corollary}[Theorem]{Corollary}
\newtheorem{Remark}[Theorem]{Remark}

\newtheorem{Problem}[Theorem]{Problem}
\newtheorem{Conjecture}[Theorem]{Conjecture}

\begin{document}

\title{Connected and outer-connected domination number of middle graphs}

\author[F. Kazemnejad]{Farshad Kazemnejad}
\address{Farshad Kazemnejad, Department of Mathematics, School of Sciences, Ilam University, 
P.O.Box 69315-516, Ilam, Iran.}
\email{kazemnejad.farshad@gmail.com}
\author[B. Pahlavsay]{Behnaz Pahlavsay}
\address{Behnaz Pahlavsay, Department of Mathematics, Hokkaido University, Kita 10, Nishi 8, Kita-Ku, Sapporo 060-0810, Japan.}
\email{pahlavsayb@gmail.com}
\author[E. Palezzato]{Elisa Palezzato}
\address{Elisa Palezzato, Department of Mathematics, Hokkaido University, Kita 10, Nishi 8, Kita-Ku, Sapporo 060-0810, Japan.}
\email{palezzato@math.sci.hokudai.ac.jp}
\author[M. Torielli]{Michele Torielli}
\address{Michele Torielli, Department of Mathematics, GI-CoRE GSB, Hokkaido University, Kita 10, Nishi 8, Kita-Ku, Sapporo 060-0810, Japan.}
\email{torielli@math.sci.hokudai.ac.jp}

\date{\today}

\begin{abstract}
 In this paper, we study the notions of connected domination number and of outer-connected domination number for middle graphs. Indeed, we obtain tight bounds for this number in terms of the order of the graph $M(G)$. We also compute the outer-connected domination number of some families of graphs such as star graphs, cycle graphs, wheel graphs, complete graphs, complete bipartite graphs and some operation on graphs, explicitly. Moreover, some Nordhaus-Gaddum-like relations are presented for the outer-connected domination number of middle graphs.
\\[0.2em]

\\[0.2em]

\noindent 
\end{abstract}
\maketitle
\section{\bf Introduction}

Domination problems and its many generalizations have been intensively studied in graph theory since 1950, see for example \cite{HHS5}, \cite{HHS6}, \cite{HeYe13}, \cite{farsom},  \cite{3totdominrook} and \cite{dominLatin}. In this paper, we use standard notation for graphs and we assume that every graph is non-empty, finite, undirected and simple. We refer to \cite{bondy2008graph} as a general reference on the subject.



Given a simple graph $G$, a \emph{dominating set} of $G$ is a set $S\subseteq V(G)$
such that  $N_G[v]\cap
S\neq \emptyset$, for any vertex $v\in V(G)$, where $N_G[v]$ is the closed neighborhood of $v$. The \emph{domination number} of $G$ is the minimum cardinality of a dominating set of $G$ and it is denoted by $\gamma(G)$. 

An important subclass of the dominating sets, that is central to this paper, is the class of connected dominating sets introduced in \cite{cds}.

\begin{Definition}
 A dominating set $S$ of a graph $G$ is called \emph{connected dominating set} if the induced subgraph $G[S]$ is connected. The minimum cardinality taken over all connected dominating sets in $G$ is called the \emph{connected domination number} of $G$ and is denoted by $\gamma_{c}(G)$. Moreover, a connected dominating set of $G$ of cardinality $\gamma_{c}(G)$ is called a $\gamma_{c}$-set of $G$.
\end{Definition}

In \cite{ocds}, the authors, taking inspiration from the notion of connected dominating set, introduced the concept of outer-connected dominating set.
\begin{Definition}
A dominating set $S$ of a graph $G$ is called an \emph{outer-connected dominating set}  if the graph $G-S$ is connected. The minimum cardinality of an outer-connected dominating set of $G$ is called the \emph{outer-connected domination number} of $G$ and it is denoted by $\gamma{\tilde{}} _{c}(G)$.
\end{Definition}

Following our previous works \cite{DSFBME}, \cite{KPPT2021color} and \cite{KPPT2021total}, the aim of this paper is to study connected dominating sets and outer-connected dominating set of middle graphs.
The concept of middle graph of a graph was first introduced in \cite{HamYos} as an intersection graph.

\begin{Definition}
 The middle graph $M(G)$ of a graph $G$ is the graph whose vertex set is $V(G)\cup E(G)$, where two vertices $x, y$ in the vertex set of $M(G)$ are adjacent in $M(G)$ if one the following holds
 \begin{enumerate}
 \item $x, y\in E(G)$ and $x, y$ are adjacent in $G$;
 \item $x\in V (G)$, $y\in E(G)$, and $x, y$ are incident in $G$. 
 \end{enumerate}
 \end{Definition}
 
 Notice that, by definition, if $G$ is a graph of order $n$ and size $m$, then $M(G)$ is a graph of order $n+m$ and size $2m+|E(L(G))| $, where $L(G)$ is the line graph of $G$.
 

In order to avoid confusion throughout the paper, we will use a ``standard'' notation for the vertex set and the edge set of the middle graph $M(G)$. In particular, if $V(G)=\{v_1,v_2,\dots, v_n\}$, then we set $V(M(G))=V(G)\cup \mathcal{M}$, where $\mathcal{M}=\{m_{ij}~|~ v_iv_j\in E(G)\}$ and $E(M(G))=\{v_im_{ij},v_jm_{ij}~|~ v_iv_j\in E(G)\}\cup E(L(G)) $.


The paper is organized as follows. In Section 2, we recall few known results on outer-connected domination numbers and domination numbers. In Section 3, we compute the connected domination number of the middle graph of a connected graph. In Section 4, we present some upper and lower bounds for $\gamma{\tilde{}} _{c}(M(G))$ in terms of the order of the graph $G$, we relate the outer-connected domination number of $M(G)$ to the edge cover number of $G$ and we compute explicitly $\gamma{\tilde{}} _{c}(M(G))$ for several known families of graphs. In Section 5, we compute the outer-connected domination number of the middle graphs of graphs obtained by some special operation. In Section 6, we present some Nordhaus-Gaddum like relations for the outer-connected domination number of middle graphs. We then conclude the paper with a section composed of open problems and conjectures.

\section{Preliminares}

In this short section, we recall three results which will be useful for our investigation.
\begin{Theorem}[\cite{ocds}]\label{outGwithdelta} 
	If $G$ is a connected graph of order $n$, then
	$$\gamma{\tilde{}} _{c}(G)\le n-\delta(G),$$
	where $\delta(G)$ is the minimum degree of a vertex in $G$.
\end{Theorem}
\begin{Theorem}[\cite{DSFBME}]\label{theo:mindominconnectgraph}
	Let $G$ be a graph with $n\ge2$ vertices. Assume $G$ has no isolated vertices, then 
	$$\lceil\frac{n}{2}\rceil\le\gamma(M(G))\le n-1.$$
\end{Theorem}
\begin{Theorem}[\cite{DSFBME}]\label{covernumber}
	Let $G$ be a graph of order $n\geq 2$ with no isolated vertex. Then $$\gamma(M(G))= \rho(G),$$
	where $\rho(G)$ is the edge cover number of $G$, i.e the minimum cardinality of an edge cover of $G$.
\end{Theorem}
\section{\bf Connected domination number of middle graphs}
In this section, we calculate the exact value of the connected domination number $\gamma _{c}(M(G))$ for any connected graph $G$ of order $n \geq 3$.
\begin{Theorem}\label{gammacMG}
For any connected graph $G$ of order $n \geq 3$,
$$\gamma _{c}(M(G))=n-1.$$
\end{Theorem}
\begin{proof}
To fix notation, let $G$ be a connected graph with vertex set $V(G)=\{v_1, \cdots, v_n\}$. Then $V((M(G)))=V(G) \cup \mathcal{M}$ where $\mathcal{M}=\{m_{ij}~|~v_iv_j \in E(G)\}$. 
First assume that $G=T$ is a tree. Obviously, $\mathcal{M}$ forms a unique minimal connected path in $M(T)$ such that $N_{M(T)}[\mathcal{M}]=V(M(T))$. This implies that $\mathcal{M}$ is the minimal connected dominating set of $M(T)$ with $|\mathcal{M}|=n-1$, and hence $\gamma _{c}(M(G))=n-1$.

Now assume that $G$ is not tree. Then consider a spanning tree $H$ of $G$, and let $\mathcal{M}_1 \subseteq \mathcal{M}$ be the vertices subdividing the edges set of $H$ in $M(H)$.  Obviously, $\mathcal{M}_1$ forms a connected path in $M(H)$ such that $N_{M(H)}[\mathcal{M}_1]=V(M(H))$. Consider $\mathcal{M}_2 =\mathcal{M}\setminus\mathcal{M}_1$. Clearly $\mathcal{M}_1$  dominates all the vertices in $\mathcal{M}_2$, and hence $N_{M(G)}[\mathcal{M}_1]=\mathcal{M}_2 \cup V(M(H))=V(M(G))$. This implies that $\mathcal{M}_1$ forms the minimal connected path in $M(G)$ with $|\mathcal{M}_1|=n-1$, and hence, $\gamma _{c}(M(G))=n-1$.
\end{proof}

\section{\bf Outer-connected domination number of middle graphs}
As we will see in this section, the computation of the outer-connected domination number is more intricate than the one for the connected domination number.

We start our study by describing a lower and an upper bound for the outer-connected domination number of the middle graph.

\begin{Theorem}\label{theo:mindomintree}
Let $G$ be a connected graph with $n\ge2$ vertices. Then $$\lceil\frac{n}{2}\rceil\le\gamma{\tilde{}} _{c}(M(G))\le n.$$
\end{Theorem}
\begin{proof} 
	If we consider $D=V(G)$, then $D$ is an outer-connected dominating set of $M(G)$, and hence, $\gamma{\tilde{}} _{c}(M(G))\le n$,  proving the second inequality.
	
	By Theorem \ref{theo:mindominconnectgraph} we have $\gamma{\tilde{}} _{c}(M(G)) \ge \lceil\frac{n}{2}\rceil$,  proving the first inequality.
	\end{proof}
As an immediate consequence of \ref{covernumber}, we have the following result.
\begin{Corollary}
	Let $G$ be a graph of order $n\geq 2$ with no isolated vertex. Then $$\gamma{\tilde{}} _{c}(M(G)) \ge \rho(G).$$
\end{Corollary}	
In the next theorem, we calculate the outer-connected domination number of the middle graph of a tree.
\begin{Theorem}\label{theo:outmindomintreeneqstar}
	Let $T$ be a tree with $n\ge2$ vertices. Then $$\gamma{\tilde{}} _{c}(M(T))= n.$$
\end{Theorem}
\begin{proof}
 Let $T$ be a tree of order $n$ with $V(T)=\{v_1, \dots, v_n\}$. Then $V(M(T))=V(T) \cup \mathcal{M}$ where $\mathcal{M}=\{m_{ij}~|~ v_iv_j\in E(T)\}$. 
 Let $\mathcal{L}=\{v_i \in V ~|~d_{T}(v_i)=1\}$ be the set of leaves of $T$ with $|\mathcal{L}|=l$, and consider
 \[
 \mathcal{M}_{1}=\{m_{ij}~|~ v_i \in \mathcal{L}~ \mbox{or}~ v_j \in \mathcal{L}\},
 \]
 
 \[
 \mathcal{M}_{2}=\mathcal{M} \setminus \mathcal{M}_{1}.
 \]

If there exists a vertex $v_i \in \mathcal{L}$ such that $v_i \notin D$, since $ N_{M(T)}[v_i]\cap D\neq \emptyset$, then $ N_{M(T)}[v_i]\cap D=\{m_{ij}\}$ for some $m_{ij} \in \mathcal{M}_{1}$. As a consequence $m_{ij} \in D$ and $v_i \notin D$, and hence $M(G)-D$ is disconnected, which is a contradiction. This implies that $\mathcal{L} \subseteq D$ and $|D \cap \mathcal{L}|=l$.

 Let $m_{ij} \in \mathcal{M}_{2}$ be such that $m_{ij} \in D$. Then, obviously  $M(G)-D$ is disconnected, which is a contradiction. As a consequence, $\mathcal{M}_{2} \cap D=\emptyset$.
 
 Now since for any $v_i \in V(T) \setminus \mathcal{L}$ we have that $ N_{M(T)}[v_i]\cap D\neq \emptyset$ and $N_{M(T)}[v_i]\cap D \subseteq \mathcal{M}_{1} \cup (V(T) \setminus \mathcal{L})$, and for every distinct $v_i, v_j \in V(T) \setminus \mathcal{L}$ we have that $(N_{M(T)}[v_i]\cap D) \cap (N_{M(T)}[v_j]\cap D)=\emptyset $, this implies that $|D \cap (\mathcal{M}_{1} \cup (V(T) \setminus \mathcal{L}))| \geq n-l$. Hence $$|D|=|D \cap \mathcal{L}|+|D \cap (\mathcal{M}_{1} \cup (V(T) \setminus \mathcal{L}))| \geq l+(n-l)=n.$$
By Theorem \ref{theo:mindomintree}, we conclude that $\gamma{\tilde{}} _{c}(M(T))= n.$

\end{proof}
\begin{Remark}
By Theorem \ref{theo:outmindomintreeneqstar}, the upper bound described in Theorem \ref{theo:mindomintree} is tight.
\end{Remark}

\begin{Corollary}\label{gamaact<gamactildt}
If $T$ is a tree of order $n$, then
$$\gamma{\tilde{}} _{c}(T) <\gamma{\tilde{}} _{c}(M(T)). $$
\end{Corollary}
\begin{proof} By Theorem \ref{theo:outmindomintreeneqstar} $\gamma{\tilde{}} _{c}(M(T))=n$. On the other hand, $\gamma{\tilde{}} _{c}(T)\leq n-1$ by Theorem \ref{outGwithdelta}, and hence we obtain the described inequality.
\end{proof}

Recall that the \emph{line graph} $L(G)$ of a graph $G$ is the graph with vertex set $E(G)$, where vertices $x$ and $y$ are adjacent in $L(G)$ if and only if the corresponding edges $x$ and $y$ share a common vertex in $G$. Directly from this definition and Theorem \ref{theo:outmindomintreeneqstar}, we obtain the following result.

\begin{Corollary}\label{LTREE<MTREE}
For any tree $T$ of order $n\geq 2$, $$\gamma{\tilde{}} _{c}(L(T)) < \gamma{\tilde{}} _{c}(M(T)).$$
\end{Corollary}
\begin{proof}
By definition $V(L(T))=E(T)$ and hence, $|V(L(T))|=n-1$. This clearly implies that $\gamma{\tilde{}} _{c}(L(T))\leq n-1$. Hence $\gamma{\tilde{}} _{c}(L(T)) \leq n-1 <n= \gamma{\tilde{}} _{c}(M(T))$ by Theorem \ref{theo:outmindomintreeneqstar}.
\end{proof}
By Theorem \ref{theo:outmindomintreeneqstar}, we can characterize the trees by looking at the outer-connected domination number of their middle graph.
\begin{Theorem}\label{outGisnwhenGistree}
Let $G$ be a connected graph with $n\ge4$ vertices. Then
$$\gamma{\tilde{}} _{c}(M(G))= n~~\text{if and only if $G$~ is a tree}.$$
\end{Theorem}
\begin{proof}
Assume that $V(G)=\{v_1, \dots, v_n\}$. Then $V(M(G))=V(G) \cup \mathcal{M}$ where $\mathcal{M}=\{m_{ij}~|~ v_iv_j\in E(G)\}$. If $G$ is a tree, then $\gamma{\tilde{}} _{c}(M(G))= n$, by Theorem \ref{theo:outmindomintreeneqstar}. 
On the other hand, assume that $\gamma{\tilde{}} _{c}(M(G))= n$ and $G$ is not tree. Then $G$ contains at least a cycle of order $n \geq 3$ as a induced subgraph. Without loss of generality, assume that $G[v_1,v_2, \dots, v_k]$ is a cycle of length $k$,  for some $k \geq 3$. 
Consider $D=\{v_3,v_4, \dots, v_n\} \cup \{m_{12}\}$. Then $D$ is an outer-connected dominating set of $M(G)$ with $|D|=n-1$, and hence $\gamma{\tilde{}} _{c}(M(G))\leq n-1$, which is a contradiction. This implies that $G$ is a tree.
\end{proof}

In the next theorem we calculate outer-connected domination number for complete graph $K_n$ where $n \ge 3$. Notice that $K_2$ is a tree and hence $\gamma{\tilde{}} _{c}(M(K_{2}))=2$ by Theorem \ref{theo:outmindomintreeneqstar}.

\begin{Theorem}\label{theo:outmindomincomplete}
For any complete graph $ K_{n}$ of order $n \geq 3$, we have
	  $$\gamma{\tilde{}} _{c}(M(K_{n}))=\lceil n/2\rceil$$
\end{Theorem}
\begin{proof}
Assume that $V(M(K_n))=V(K_n) \cup \mathcal{M}$ where $V(K_n)=\{v_1, \dots, v_n\}$ and $\mathcal{M}=\{m_{ij} ~|~v_i v_j \in E(G)\}$. 
When $n=3$, it is easy to check that $\gamma{\tilde{}} _{c}(M(K_{n}))=2$, by considering $D=\{v_1, m_{23}\}$. Now let $n \geq 4$. Assume that $n$ is even and consider \[
D=\{m_{12}, m_{34}, \dots, m_{(n-1)n}\}
.\] Then $D$ is an outer-connected dominating set of $M(G)$ with $|D|=\lceil n/2\rceil$. Similary, if $n$ is odd, consider 
\[D=\{m_{12}, m_{34}, \dots, m_{(n-2)(n-1)}, m_{(n-1)n}\}.\]

 Then $D$ is an outer-connected dominating set of $M(G)$ with $|D|=\lceil n/2\rceil$. This show that $\gamma{\tilde{}} _{c}(M(K_{n}))\leq \lceil n/2\rceil .$ On the other hand, by Theorem \ref{theo:mindomintree}, $\gamma{\tilde{}} _{c}(M(K_{n}))\geq \lceil n/2\rceil .$
 \\
\end{proof}
\begin{Remark}
By Theorem \ref{theo:outmindomincomplete}, the lower bound described in Theorem \ref{theo:mindomintree} is tight.
\end{Remark}

\begin{Theorem}\label{theomoutdomcycle}
For any cycle $C_n$ of order $n \geq 3$,  $$\gamma{\tilde{}} _{c}(M(C_{n}))=n-1.$$
\end{Theorem}
\begin{proof}
To fix the notation, assume that $V(C_n)=\{v_1, \dots, v_n\}$ and $E(C_n)=\{v_1v_2, v_2v_3,\dots,v_{n-1}v_n,v_nv_1\}$. Then $V(M(C_n))=V(C_n)\cup \mathcal{M}$, where $\mathcal{M}=\{ m_{i(i+1)}~|~1\leq i \leq n-1 \}\cup\{m_{1n}\}$.
Consider $D=\{v_3 , \dots, v_n\} \cup \{m_{12}\}$, then $D$ is an outer-connected dominating set with $|D|=n-1$, and hence $\gamma{\tilde{}} _{c}(M(C_{n}))\le n-1.$

 Let $D$ be a minimal outer-connected dominating set of $M(C_{n})$.




If $\mathcal{M} \cap D=\emptyset$, then $N_{M(C_{n})}[v_i]\cap D \neq \emptyset$ for every $1 \leq i \leq n$, implies that $V \subseteq D$ and hence $|D| \geq n$, contradicting the minimality of $D$. 




Assume that $|\mathcal{M} \cap D|=1$. Without loss of generality, we can assume that $m_{12} \in D$. Then $N_{M(C_{n})}[v_i]\cap D \neq \emptyset$ for  $ i\neq 1,2$ implies that $\{v_3, \dots, v_n\} \subseteq D$ and so $\gamma{\tilde{}} _{c}(M(C_{n})) \geq n-1$, proving our statement.




Assume that $|\mathcal{M} \cap D|=2$. Let $m_{ij}, m_{pq} \in D$ for some $i,j,p,q$. First, assume that $m_{ij}$ is adjacent to $m_{pq}$ in $M(C_n)$. Without loss of generality, we can assume that $m_{12}, m_{23} \in D$. Since $M(C_n)-D$ is connected, then $v_2 \in D$. Moreover, $N_{M(C_{n})}[v_i]\cap D \neq \emptyset$ for  $ i\neq 1,2,3$, implies that $\{v_4, \dots, v_n\} \subseteq D$ and hence $|D| \geq (n-4+1)+3=n$, contradicting the minimality of $D$.
Assume now that $m_{ij}$ is non-adjacent to $m_{pq}$ in $M(C_n)$ with $i<j<p<q$.
Since $m_{ij}, m_{pq} \notin N_{M(C_{n})}[v_k]$ for every $k \in \{ 1, 2, \dots, n\}\setminus\{i,j,p,q\}$, then $|D \cap V(C_n)| \geq n-4$. On the other hand, since $M(C_n)-D$ is connected then $v_j, v_p \in D$. This implies that $|D|=|D \cap V(C_n)|+|D \cap \mathcal{M}| \geq (n-2)+2=n$, contradicting the minimality of $D$.



Assume that $|\mathcal{M} \cap D|=k \geq 3$. Since $M(C_n)-D$ is connected, then $k<n-1$. Without loss of generality, we can assume that $\mathcal{M} \cap D=\{m_{i_1(i_1+1)},\dots,m_{i_k(i_k+1)}\}$ where $i_1<i_2<\cdots<i_k$. If $i_j+1<i_{j+1}$ for some $1\le j\le k-1$, then $M(C_n)-D$ is disconnected. This implies that $i_j+1=i_{j+1}$ for all $1\le j\le k-1$. Let $I=\{i,i+1~|~m_{i(i+1)} \in D\}$, $V_1=\{v_i \in V~|~ i \notin I\}$. $N_{M(C_{n})}[v_i]\cap D \neq \emptyset$ for  $v_i \in V_1$, implies that $V_1 \subseteq D$. Moreover, since $M(C_n)-D$ is connected, $v_{i_1+1}, \dots, v_{i_k}\in D$.
As a consequence, $D=\{m_{i_1(i_1+1)},\dots,m_{i_k(i_k+1)}\}\cup V_1\cup \{v_{i_1+1}, \dots, v_{i_k}\}$. This implies that $|D|=k+(n-k-1)+k-1=n+k-2\ge n+1$, contradicting the minimality of $D$.

Therefore, $\gamma{\tilde{}} _{c}(M(C_{n}))= n-1.$
\end{proof}

\begin{Theorem} \label{OUTERWN} For any wheel $W_n$ of order $n\geq 4$, 
	$$\gamma{\tilde{}} _{c}(M(W_{n}))= \lceil n/2\rceil .$$
\end{Theorem}
\begin{proof}
	To fix the notation, assume $V(W_n)=V=\{v_0,v_1,\dots, v_{n-1}\}$ and $E(W_n)=\{v_0v_1,v_0v_2,\dots, v_0v_{n-1}\}\cup\{v_1v_2, v_2v_3,\dots,v_{n-1}v_1\}$. Then we have $V(M(W_n))=V(W_n)\cup \mathcal{M}$, where $\mathcal{M}=\{ m_{0i}~|~1\leq i \leq n-1 \}\cup\{ m_{i(i+1)}~|~1\leq i \leq n-2 \}\cup\{m_{1(n-1)}\}$. 
	Assume that $n$ is even and consider $D=\{m_{12}, m_{34}, \dots, m_{(n-3)(n-2)}\} \cup \{m_{0(n-1)}\}$. Then $D$ is an outer-connected dominating set of $M(G)$ with $|D|=\lceil n/2\rceil$. Similarly, if $n$ is odd, consider $D=\{m_{12}, m_{34}, \dots, m_{(n-2)(n-1)}\} \cup \{m_{0(n-1)}\}$. Then $D$ is an outer-connected dominating set of $M(G)$ with $|D|=\lceil n/2\rceil$. This show that $\gamma{\tilde{}} _{c}(M(W_{n}))\leq \lceil n/2\rceil$. On the other hand, by Theorem \ref{theo:mindomintree}, $\gamma{\tilde{}} _{c}(M(W_{n}))\geq \lceil n/2\rceil .$
\end{proof}
\begin{Theorem}\label{outcompletebipartitegr}
	Let $K_{n_1,n_2}$ be the complete bipartite graph with $n_2\geq n_1 \geq 2$. Then $$\gamma{\tilde{}} _{c}(M(K_{n_1,n_2}))= n_2.$$
\end{Theorem}
\begin{proof}  To fix the notation, assume $V(K_{n_1,n_2})=\{v_1,\dots, v_{n_1},u_1,\dots, u_{n_2}\}$ and $E(K_{n_1,n_2})=\{v_iu_j ~|~1\leq i \leq n_1, 1\leq j \leq n_2\}$. Then $V(M(K_{n_1,n_2}))=V(K_{n_1,n_2})\cup \mathcal{M}$, where
	$\mathcal{M}=\{ m_{ij}~|~1\leq i \leq n_1, 1\leq j \leq n_2 \}. $
Let $D$ be an outer-connected dominating set of $M(K_{n_1,n_2})$. 	
Since $D$ is a dominating set for $M(K_{n_1,n_2})$, it has to dominate $u_1,\dots, u_{n_2}$ that are all disconnected. This implies that $\gamma(M(K_{n_1,n_2}))\ge n_2$. Now since $D=\{m_{11},m_{22},\dots, m_{n_1n_1}\}\cup\{m_{n_1(n_1+1)}, m_{n_1(n_1+2)},\dots, m_{n_1n_2}\}$ is an outer-connected dominating set of $M(K_{n_1,n_2})$ with
	 $|D|=n_1+n_2-n_1=n_2$, This implies that 
	 $\gamma{\tilde{}} _{c}(M(K_{n_1,n_2}))= n_2.$
\end{proof}




\begin{Theorem}\label{outfriendship}
	Let $F_n$ be the friendship graph with $n\ge2$. Then $$\gamma{\tilde{}} _{c}(M(F_n))= n+1.$$
\end{Theorem}
\begin{proof} To fix the notation, assume $V(F_n)=\{v_0,v_1,\dots, v_{2n}\}$ and $E(F_n)=\{v_0v_1,v_0v_2,\dots, v_0v_{2n}\}\cup\{v_1v_2, v_3v_4,\dots,v_{2n-1}v_{2n}\}$. Then $V(M(F_n))=V(F_n)\cup \mathcal{M}$, where $\mathcal{M}=\{ m_i~|~1\leq i \leq 2n \}\cup\{ m_{i(i+1)}~|~1\leq i \leq 2n-1 \text{ and } i \text{ is odd}\}$.

By Theorem~\ref{theo:mindominconnectgraph}, $\gamma(M(F_n))\ge \lceil\frac{2n+1}{2} \rceil=n+1$. Now since  $D=\{ m_{i(i+1)}~|~1\leq i \leq 2n-1 \text{ and } i \text{ is odd}\}\cup\{v_0\}$ is an outer-connected dominating set for $M(F_n)$ with $|D|=n+1$, we have $\gamma{\tilde{}} _{c}(M(F_n))= n+1.$  
\end{proof}

Putting together Theorems \ref{theomoutdomcycle} and \ref{gammacMG}, we have the following result.
\begin{Corollary}
	There exists a connected graph $G$ of order $n \geq 3$ such that
	$$\gamma _{c}(M(G))=\gamma{\tilde{}} _{c}(M(G)).$$
\end{Corollary}
\begin{Remark}
Comparing Theorem \ref{gammacMG} and Theorem \ref{theo:outmindomintreeneqstar} we conclude that for any tree $T$ we have
$$\gamma _{c}(M(T))<\gamma{\tilde{}} _{c}(M(T)).$$ 

Similarly, comparing Theorems \ref{gammacMG} and \ref{theomoutdomcycle}  we conclude that for any cycle $$\gamma _{c}(M(C_n))=\gamma{\tilde{}} _{c}(M(C_n)).$$

Finally, comparing Theorem \ref{gammacMG} and Theorems \ref{theo:outmindomincomplete} and \ref{OUTERWN} we conclude that $$\gamma{\tilde{}} _{c}(M(K_n))=\gamma{\tilde{}} _{c}(M(W_n))<\gamma _{c}(M(K_n))=\gamma _{c}(M(W_n)).$$ 

As a consequence, if $G$ be a connected graph of order $n$, then one may not conclude that
\[
\gamma _{c}(M(G)) \geq \gamma{\tilde{}} _{c}(M(G))~~~~~\mbox{or}~~~~~~\gamma _{c}(M(G)) \leq \gamma{\tilde{}} _{c}(M(G)).
\]
\end{Remark}

\section{Operation on graphs}

In this section, we study the outer-connected domination number for the middle graph of the corona, $2$-corona and other types of graphs.
\begin{Definition}
	The \emph{corona} $G\circ K_1$ of a graph $G$ is the graph of order $2|V(G)|$ obtained from $G$ by adding a pendant edge to each vertex of $G$. The \emph{$2$-corona} $G\circ P_2$ of $G$ is the graph of order $3|V(G)|$ obtained from $G$ by attaching a path of length $2$ to each vertex of $G$ so that the resulting paths are vertex-disjoint.
\end{Definition}

\begin{Theorem}\label{theo:minoutdomincorona}
	For any connected graph $G$ of order $n\geq 2$, $$n+\lceil n/2 \rceil \leq \gamma{\tilde{}} _{c}(M(G\circ K_1)) \leq 2n.$$
\end{Theorem}
\begin{proof} 
	Assume $V(G)=\{v_1,\dots, v_{n}\}$, then $V(G\circ K_1)= \{v_{1},\dots, v_{2n}\}$ and $E(G\circ K_1)=\{v_1v_{n+1},\dots, v_nv_{2n} \}\cup E(G) $. As a consequence, $V(M(G\circ K_1))=V(G\circ K_1)\cup \mathcal{M}$, where $\mathcal{M}=\{ m_{i(n+i)}~|~1\leq i \leq n \}\cup \{ m_{ij}~|~v_iv_j\in  E(G)\}$.
	Since $\{v_1,\dots, v_{2n}\}$ is an outer-connected dominating set of $M(G\circ K_1)$, we have $\gamma{\tilde{}} _{c}(M(G\circ K_1)) \leq 2n$.

	 Let $D$ be an outer-connected dominating set of $M(G\circ K_1)$. 
	 Assume $v_{n+i} \notin D$ for some $1 \le i \le n$, then since $D$ is a dominating set of $M(G\circ K_1)$ this implies that $m_{i(n+i)} \in D$ and so $M(G\circ K_1)-D$ is disconnected, which is a contradiction. As a consequence $D_1=\{v_{n+1},\dots, v_{2n}\}\subseteq D$. Now since $N_{M(G\circ K_1)}[v] \cap D_{1}=\emptyset$ for all $v \in V(M(G))$, by Theorem \ref{theo:mindominconnectgraph}, we have $$\gamma{\tilde{}} _{c}(M(G\circ K_1)) \geq n+\gamma(M(G)) \geq n+\lceil n/2 \rceil.$$

\end{proof}
\begin{Remark}
The upper bound in Theorem \ref{theo:minoutdomincorona} is tight. In fact, when $G$ is a tree, then $G\circ K_1$ is also a tree and so $\gamma{\tilde{}} _{c}(M(G\circ K_1))=2n$ by Theorem \ref{theo:outmindomintreeneqstar}.

Moreover, also the lower bound in Theorem \ref{theo:minoutdomincorona} is tight. To see it, consider $G=K_n$ and 
$$D=\{v_{n+i}~|~1 \leq i \leq n\} \cup \{m_{12}, m_{34}, \dots, m_{(n-1)n}\}$$	when $n$ is even, and 
$$D=\{v_{n+i}~|~1 \leq i \leq n\} \cup \{m_{12}, m_{34}, \dots, m_{(n-2)(n-1)} m_{(n-1)n}\}$$	when $n$ is odd. In each case, $D$ is an outer-connected dominating set of $M(K_n\circ K_1)$ with $|D|=n+\lceil n/2 \rceil$.
\end{Remark}

Similarly to Theorem \ref{theo:minoutdomincorona}, we can describe lower and upper bounds for the outer-connected domination number of the middle graph of a $2$-corona graph.
\begin{Theorem}\label{theominoutdomin2corona}
	For any connected graph $G$ of order $n\geq 2$, $$2n+\lceil n/2 \rceil \leq \gamma{\tilde{}} _{c}(M(G\circ P_2)) \leq 3n.$$
\end{Theorem}
\begin{proof} 
	Assume $V(G)=\{v_1,\dots, v_n\}$, then $V(G\circ P_2)= \{v_{1},\dots, v_{3n}\}$ and $E(G\circ P_2)=\{v_iv_{n+i}, v_{n+i}v_{2n+i}~|~1\leq i \leq n \}\cup E(G) $. As a consequence, we have that $V(M(G\circ P_2))=V(G\circ P_2)\cup \mathcal{M}$, where $\mathcal{M}=\{ m_{i(n+i)},m_{(n+i)(2n+i)}~|~1\leq i \leq n \}\cup \{ m_{ij}~|~v_iv_j\in  E(G)\}$. 
	Since $\{v_1,\dots, v_{3n}\}$ is an outer-connected dominating set of $M(G\circ P_2)$, we have $\gamma{\tilde{}} _{c}(M(G\circ P_2)) \leq 3n$.
	
	 Let $D$ be an outer-connected dominating set of $M(G\circ P_2)$. To prove first inequality, we claim that $$|D_1|=|\{v_{n+i},v_{2n+i},m_{(n+i)(2n+i)}~|~1\leq i \leq n \} \cap D | \ge 2n$$
	 Assume $v_{2n+i} \notin D$ for some $1 \le i \le n$. Since $D$ is a dominating set of $M(G\circ P_2)$ this implies that $m_{(n+i)(2n+i)} \in D$ and so $M(G\circ P_2)-D$ is disconnected, which is a contradiction. Hence $\{v_{2n+i}~|~1\leq i \leq n \} \subseteq D$. Now assume  $v_{n+i} \notin D$ for some $1 \le i \le n$. Since $D$ is a dominating set of $M(G\circ P_2)$ this implies that $m_{(n+i)(2n+i)} \in D$ or $m_{i(n+i)} \in D$. If $m_{i(n+i)} \in D$, then $M(G\circ P_2)-D$ is disconnected, which is a contradiction, and hence $m_{(n+i)(2n+i)}\in D$. This shows that for every $1 \le i \le n$, we have that $v_{n+i} \in D$ or $m_{(n+i)(2n+i)}\in D$.
	 Now since by construction of $D_1$, we have that $N_{M(G\circ P_2)}[v] \cap D_{1}=\emptyset$ for all $v \in V(M(G))$, this implies $$\gamma{\tilde{}} _{c}(M(G\circ P_2)) \geq 2n+\gamma(M(G)) \geq 2n+\lceil n/2 \rceil$$ by Theorem \ref{theo:mindominconnectgraph}.

\end{proof}
\begin{Remark}
	The upper bound in Theorem \ref{theominoutdomin2corona} is tight. This is because when $G$ is a tree, then also $G\circ P_2$ is a tree and hence $\gamma{\tilde{}} _{c}(M(G\circ P_2))=3n$ by Theorem \ref{theo:outmindomintreeneqstar}.

	Moreover, also the lower bound in Theorem \ref{theominoutdomin2corona} is tight. Consider $G=K_n$, and
	$$D=\{v_{n+i},v_{2n+i}~|~1 \leq i \leq n\} \cup \{m_{12}, m_{34}, \dots, m_{(n-1)n}\}$$	when $n$ is even, and 
	$$D=\{v_{n+i},v_{2n+i}~|~1 \leq i \leq n\} \cup \{m_{12}, m_{34}, \dots, m_{(n-2)(n-1)} m_{(n-1)n}\}$$	when $n$ is odd. In both cases, $D$ is an outer-connected dominating set of $M(K_n\circ P_2)$ with $|D|=2n+\lceil n/2 \rceil$.
\end{Remark}

In the next two theorems, we study the outer-connected domination number of the middle graph of the join of a graph with $\overline{K_p}$. 
\begin{Theorem}\label{theomindominjoinpbig}
	For any connected graph $G$ of order $n\geq 2$ and any integer $p\geq n$, 
	 $$\gamma{\tilde{}} _{c}(M(G+\overline{K_p}))=p.$$
\end{Theorem}
\begin{proof}
	Assume $V(G)=\{v_1,\dots,v_n\}$ and $V(\overline{K_p})=\{v_{n+1},\dots,v_{n+p}\}$. Then $V(M(G+\overline{K_p}))=V(G+\overline{K_p})\cup  \mathcal{M}_1\cup  \mathcal{M}_2$ where $\mathcal{M}_1= \{m_{ij}~|~v_iv_j\in  E(G)\}$ and $\mathcal{M}_2= \{m_{i(n+j)}~|~1\leq i \leq n, 1\leq j \leq p\}$.  
	
	By  \cite[Theorem 2.15]{DSFBME}, we have that $\gamma(M(G+\overline{K_p}))= p$, and hence $\gamma{\tilde{}} _{c}(M(G+\overline{K_p})) \geq \gamma(M(G+\overline{K_p}))= p$.
	
	On the other hand, if we consider $D=\{m_{i(n+i)}~|~1\leq i \leq n\}\cup \{m_{1(n+j)}~|~n+1\leq j \leq p\}$, then $D$ is an outer-connected dominating set of $M(G+\overline{K_p})$ with $|D|=p$, and hence $\gamma{\tilde{}} _{c}(M(G+\overline{K_p})) \le p$.
\end{proof}

\begin{Theorem}\label{theo:mindominjoinpsmallineq}
	For any connected graph $G$ of order $n\geq 2$ and any integer $p<n$, 
	$$\lceil\frac{n+p}{2}\rceil\le\gamma{\tilde{}} _{c}(M(G+\overline{K_p}))\le n.$$
\end{Theorem}
\begin{proof} The first inequality follows directly from Theorem~\ref{theo:mindomintree}.
	On the other hand, using the same notation as in the proof of Theorem~\ref{theomindominjoinpbig}, if we consider $D=\{m_{i(n+i)}~|~1\le i\le p\}\cup\{v_j~|~p+1\le j\le n\}$, then $D$ is an outer-connected dominating set of $M(G+\overline{K_p})$ with $|D|=n$, and hence we obtain the second inequality.
\end{proof}

\begin{Remark} Both inequalities in Theorem~\ref{theo:mindominjoinpsmallineq} are sharp. In fact, if we consider $G=C_4$ and $p=2$, then a direct computation shows that $\gamma{\tilde{}} _{c}(M(C_4+\overline{K_2}))=3=\lceil\frac{n+p}{2}\rceil$. Similarly, if we consider $G=C_4$ and $p=3$, then $\gamma{\tilde{}} _{c}(M(C_4+\overline{K_3}))=4=n$. 
\end{Remark}
\section{Nordhaus-Gaddum-like relations}

Finding a Nordhaus-Gaddum-like relation for any parameter in graph theory is one of the traditional works which started after the following theorem by Nordhaus and Gaddum from \cite{Nordhaus}.

\begin{Theorem}[\cite{Nordhaus}]\label{Nordhaus and Gaddum}
	For any graph $G$ of order $n$, $2\sqrt{n}\leq \chi(G)+\chi(\overline{G}) \leq n+1$.
\end{Theorem}

In this section, we find Nordhaus-Gaddum-like relations for the outer-connected  domination number of middle graphs.
In particular, by Theorems \ref{theo:mindomintree} and \ref{gammacMG}, we have the following result.
\vskip 0.02cm

\begin{Corollary} \label{gammacandoutgamma}
	Let $G$ be a connected graph with $n\ge4$ vertices, where $G$ is not tree. Then $$n+\lceil\frac{n}{2}\rceil-1 \le \gamma _{c}(M(G))+\gamma{\tilde{}} _{c}(M(G))\le 2n-2,$$
	$$\lceil\frac{n}{2}\rceil(n-1) \le \gamma _{c}(M(G))\cdot\gamma{\tilde{}} _{c}(M(G))\le (n-1)^{2}.$$
\end{Corollary}
\begin{Remark}
	The upper bounds in Corollary \ref{gammacandoutgamma} are both tight, for example when $G$ is a cycle, by Theorem \ref{theomoutdomcycle} and Theorem \ref{gammacMG}.

	Similarly, also the lower bounds in Corollary \ref{gammacandoutgamma} are tight, for example when $G$ is a complete graph $K_n$ or a wheel graph $W_n$, by Theorems \ref{theo:outmindomincomplete}, \ref{OUTERWN} and \ref{gammacMG}.
\end{Remark}
\section{Open problems}
We conclude the paper with a series of observations and open problems related to the notion of outer-connected domination number.

By Corollary \ref{gamaact<gamactildt}, if $G$ is a tree of order $n$, then
$\gamma{\tilde{}} _{c}(G) <\gamma{\tilde{}} _{c}(M(G)).$ On the other hand, by Theorems \ref{theo:outmindomincomplete}, \ref{theomoutdomcycle}, \ref{OUTERWN}, \ref{outcompletebipartitegr}, \ref{outfriendship} and \cite{ocds}, it is easy to see that
\[
1=\gamma{\tilde{}} _{c}(K_n)=\gamma{\tilde{}} _{c}(W_n)<\gamma{\tilde{}} _{c}(M(K_n))=\gamma{\tilde{}} _{c}(M(W_n))=\lceil n/2\rceil,
\]
\[1=\gamma{\tilde{}} _{c}(F_n)<\gamma{\tilde{}} _{c}(M(F_n))=n+1, 
\]
\[2=\gamma{\tilde{}} _{c}(K_{n_1, n_2})<\gamma{\tilde{}} _{c}(M(K_{n_1, n_2}))=n_2 
\]
and
\[n-2=\gamma{\tilde{}} _{c}(C_n)<\gamma{\tilde{}} _{c}(M(C_n))=n-1. 
\]
These facts all support the following conjecture.
\begin{Conjecture}
Let $G$ be a graph of order $n\ge 2$. Then
$$\gamma{\tilde{}} _{c}(G) <\gamma{\tilde{}} _{c}(M(G)).$$
\end{Conjecture}
Similarly to the previous conjecture, it is natural to compare the outer-connected domination number of the middle graph and of the line graph.

By Corollary~\ref{LTREE<MTREE}, if $T$ is a tree, then $\gamma{\tilde{}} _{c}(L(T)) < \gamma{\tilde{}} _{c}(M(T))$. On the other hand we can obtain similar results for some known families.

\begin{Proposition}\label{LCN<MCN}
For any cycle $C_n$ of order $n \geq 3$,
 $$\gamma{\tilde{}} _{c}(L(C_n)) < \gamma{\tilde{}} _{c}(M(C_n)).$$
\end{Proposition}
\begin{proof}
By definition of line graph, $C_n$ is isomorphic to $L(C_n)$ for every $n \geq 3$. This implies that $\gamma{\tilde{}} _{c}(C_n)=\gamma{\tilde{}} _{c}(L(C_n))=n-2$ by \cite{ocds}. On the other hand $\gamma{\tilde{}} _{c}(M(C_n))=n-1$ by Theorem \ref{theomoutdomcycle}, and hence $\gamma{\tilde{}} _{c}(L(C_n)) < \gamma{\tilde{}} _{c}(M(C_n)).$
\end{proof}

\begin{Proposition} \label{OUTERWNLWN} For any wheel $W_n$ of order $n\geq 5$, 
	$$\gamma{\tilde{}} _{c}(L(W_n)) < \gamma{\tilde{}} _{c}(M(W_n)).$$
\end{Proposition}
\begin{proof}
	Let $V(W_n)=\{v_0,v_1,\dots, v_{n-1}\}$ and $E(W_n)=\{v_0v_1,\dots, v_0v_{n-1}\}\cup\{v_1v_2, v_2v_3,\dots,v_{n-1}v_1\}$. Then $V(M(W_n))=V(W_n)\cup \mathcal{M}$, where $\mathcal{M}=\{ m_{0i}~|~1\leq i \leq n-1 \}\cup\{ m_{i(i+1)}~|~1\leq i \leq n-2 \}\cup\{m_{1(n-1)}\}=V(L(W_n))$ and $E(L(W_n))=\{m_{ij} m_{pq}| \{i,j\} \cap \{p,q\}=1\}$.

	
Assume that $n$ is even and consider $D=\{m_{(2i+1)(2i+2)}| 1 \leq i \leq \lceil n/2 \rceil-2 \} \cup \{m_{01}\}$. Then $D$ is an outer-connected dominating set of $L(G)$ with $|D|=\lceil n/2\rceil-1$. Similarly, if $n$ is odd, consider $D=\{m_{0(2i+1)}|1 \leq i \leq \lceil n/2 \rceil-1 \}$. Then $D$ is an outer-connected dominating set of $L(G)$ with $|D|=\lceil n/2\rceil-1$. This show that $\gamma{\tilde{}} _{c}(M(L_{n}))\leq \lceil n/2\rceil-1$.  By Theorem \ref{OUTERWN}, $\gamma{\tilde{}} _{c}(L(W_n)) \leq \lceil n/2 \rceil-1  < \lceil n/2 \rceil= \gamma{\tilde{}} _{c}(M(W_n)).$
\end{proof}
\begin{Proposition}\label{LW4=MW4}
There exists a connected graph $G$ of order $n=4$ such that
$$\gamma{\tilde{}} _{c}(L(G)) = \gamma{\tilde{}} _{c}(M(G)).$$
\end{Proposition}
\begin{proof}
Consider $G=W_4$ with $V(G)=\{v_0, v_1,v_2,v_3\}$ and $E(G)=\{v_0 v_1,v_0 v_2,v_0 v_3, v_1 v_2,v_2 v_3,v_1 v_3\}.$ Then $V(M(G))=V \cup \mathcal{M}$ where $\mathcal{M}=\{m_{ij}~|~ v_iv_j\in E(G)$ and $V(L(G))= \mathcal{M}$. Assume that $D$ is a dominating set of $L(G)$ with $|D|=1$. Then there exists an index $i$ for some $1\leq i \leq 3$ such that $N_{L(G)}[m_{ij}] \cap D=\emptyset$ which is a contradiction. This implies that $\gamma(L(G)) \geq 2$, and hence that $\gamma{\tilde{}} _{c}(L(G)) \geq 2$. Now since $D=\{m_{12}, m_{03}\}$  is an outer-connected dominating set of $L(G)$ with $|D|=2$, we have $\gamma{\tilde{}} _{c}(L(G))=2$. By Theorem \ref{OUTERWN}
$\gamma{\tilde{}} _{c}(L(G)) = \gamma{\tilde{}} _{c}(M(G))=2.$
\end{proof}
As a consequence of Proposition~\ref{LW4=MW4}, it is natural to ask the following
\begin{Problem}
Can we classify the graphs $G$ such that 
$$\gamma{\tilde{}} _{c}(L(G)) = \gamma{\tilde{}} _{c}(M(G))?$$
\end{Problem}
In addition, the previous results also all support the following conjecture.
\begin{Conjecture}
	Let $G$ be a graph of order $n\ge 2$. Then
	$$\gamma{\tilde{}} _{c}(L(G)) \le \gamma{\tilde{}} _{c}(M(G)).$$
\end{Conjecture}
\paragraph{\textbf{Acknowledgements}} During the preparation of this article the fourth author was supported by JSPS Grant-in-Aid for Early-Career Scientists (19K14493).


\end{document}